\documentclass[12pt]{amsart}
\usepackage[T1]{fontenc}
\usepackage{amsmath}
\usepackage{amsfonts}
\usepackage{amsthm}
\usepackage{amssymb}
\usepackage{amscd}
\usepackage{epic}
\usepackage{eepic}
\usepackage{bigints}
\usepackage{enumitem,linegoal}
\usepackage{calc}
\usepackage[T1]{fontenc}
\usepackage[ansinew]{inputenc}

\newcommand{\printname}[1]
  {\smash{\makebox[0pt]{pace{-2.0in}\raisebox{8pt}{\tiny #1}}}}
\newtheorem {Lemma} {Lemma} [section]

\newtheorem {Theorem}[Lemma]{Theorem}

\newtheorem {Proposition}[Lemma]{Proposition}

\theoremstyle{definition}

\theoremstyle{remark}

\newcommand {\cal}[1]  {{\mathcal{#1}}}

\newcommand{\nab}[1][]{\ensuremath{\mathrm{\nabla}{#1}}}

\newcommand{\be}{\begin{equation}}
\newcommand{\ee}{\end{equation}}

\newcommand*\xbar[1]{%
  \hbox{%
    \vbox{%
      \hrule height 0.5pt 
      \kern0.5ex
      \hbox{%
        \kern-0.1em
        \ensuremath{#1}%
        \kern-0.1em
      }%
    }%
  }%
}

\newcommand{\lb}[1]{\label{#1}}

\newcommand{\ti}{\tilde}

\newcommand{\bet}{\beta}
\newcommand{\al}{\alpha}

\newcommand{\we}{\wedge}

\newcommand{\sig}{\sigma}

\newcommand{\lam}{\lambda}

\newcommand{\thet}{\theta}
\newcommand{\Ref}[1]{(\ref{#1})}

\newcommand{\ri}{\mathrm{Ric}}

\newcommand{\Lap}{\Delta}

\newcommand{\wht}{\widehat}

\def\t{\tau}

\newcommand{\Ll}{\cal{L}}

\newcommand{\ol}     {\overline}

\begin{document}
\title[Conformally-K\"ahler Ricci solitons and Ricci quasi-solitons]{Conformally-K\"ahler Ricci solitons and base metrics for warped product Ricci solitons}
\author{Gideon Maschler}






\renewcommand{\theequation}{\arabic{section}.\arabic{equation}}
\thispagestyle{empty}
\setcounter{equation}{0}
\address{Department of Mathematics and Computer Science, Clark University,
Worcester, Massachusetts 01610, U.S.A.}
\email{gmaschler@clarku.edu}

\begin{abstract}
We investigate K\"ahler metrics conformal to gradient Ricci solitons, and base metrics
of warped product gradient Ricci solitons. The latter we name quasi-solitons.
A main assumption that is employed is functional dependence of the soliton potential,
with the conformal factor in the first case, and with the warping function in the second.
The main result in the first case is a partial classification in dimension $n\geq 4$.
In the second case, K\"ahler quasi-soliton metrics satisfying the above main assumption
are shown to be, under an additional genericity hypothesis, necessarily Riemannian products. Another theorem
concerns quasi-soliton metrics satisfying the above main assumption, which are also conformally
K\"ahler. With some additional assumptions it is shown that such metrics are necessarily
base metrics of Einstein warped products, that is, quasi-Einstein.
\end{abstract}

\maketitle

\section{Introduction}
The study of the Ricci flow \cite{ham} has inspired the introduction of a metric type
generalizing the Einstein condition. A gradient Ricci soliton is a Riemannian
metric satisfying $$\text{$\ri+\nab df =\lam g$,\quad $\lam$ constant.}$$
The function $f$ is called the soliton potential. Such solitons are further referred to
as shrinking, steady or expanding, depending on the sign of $\lam$.

In this paper we consider Ricci solitons in two settings: the case where they are
conformal to K\"ahler metrics, and the case where they are warped products.
Conformal classes of Ricci solitons have been studied recently in \cite{condif,consol}.
K\"ahler metrics in such a conformal class, with nontrivial conformal factor, have been
examined in \cite{confsol,and-pol}. Warped product Ricci solitons, on the other hand, have
been studied extensively when the base of the warped product is one dimensional
(cf. \cite{many}). For example, the cigar soliton and the Bryant soliton belong
to this category.

In each case we focus on an auxiliary metric which at least partially
determines the soliton. In the first case that would be the associated K\"ahler metric
in the conformal class, and in the second case it is the induced metric on the base of
the warped product. We call the latter metric a (gradient Ricci) quasi-soliton, in analogy
with how base metrics of Einstein warped products are often called quasi-Einstein metrics.
We consider only quasi-soliton metrics which are K\"ahler, or conformally K\"ahler.

A common thread for these two cases of auxiliary metrics is the appearance of two Hessians
in their defining equation. One of these is the Hessian of the soliton potential $f$, while the
other Hessian depends on the case: it is that of the conformal factor $\t$ in the first case, and
that of the the warping function $\ell$ in the second.


These equations are, of course, more complex than the original Ricci soliton equation, and
handling them in full generality still appears beyond reach. Our strategy is thus to consider
mainly the case where functional dependence of the above two functions holds,
in either setting. In other words, we require \be\lb{fun-ind}\text{$d\t\we df=0$\quad in the first case, and\quad
$d\ell\we df=0$\quad in the second}.\end{equation} In the latter case we call the metric  a {\em special} quasi-soliton.

An example where the first of these conditions occurs in the K\"ahler conformally-soliton case, is when
the conformal factor $\t$ is additionally a potential for a Killing vector field of the K\"ahler metric
(a Killing potential). The latter condition has been studied in \cite{confsol} and plays
a role in Theorem \ref{quas2}. It turns out that the first of Conditions \Ref{fun-ind} also implies,
generically, the existence of a Killing potential which, however, is of a more general
kind, being only functionally dependent on $\t$, rather than being $\t$ itself.
An instance of this more general setting has first been considered in \cite{and-pol}.

Another metric type that plays an important role in all our main theorems is the
SKR metric, i.e. a metric that admits a so-called special K\"ahler-Ricci potential.
This notion that was first introduced in \cite{dr-ma1, skrp} for the purpose of
classifying conformally-K\"ahler Einstein metrics.
In all our main theorems the argument yields a Ricci-Hessian equation of the form
$$\al\nab d\t+\ri=\gamma g,$$ for functions $\al$ and $\gamma$. The theory of SKR
metrics which is then applied is closely tied to such equations.

The main results in this article are Theorems \ref{classf}, \ref{quas1} and \ref{quas2}.
The content of the first of these is a partial classification of K\"ahler metrics
conformal to gradient Ricci solitons in dimension $n\geq 4$ with the first of the
conditions in \Ref{fun-ind}. Theorem \ref{quas1} presents a reducibility
result for special quasi-soliton metrics which are K\"ahler. The conclusion of this theorem,
that the metric is a Riemannian product, is analogous to a similar result for quasi-Einstein metrics
\cite{quasE}. Theorem \ref{quas2} mixes the two main themes of this paper, as it involves special
quasi-soliton metrics that are conformal to an irreducible K\"ahler metric. With some additional
assumptions, the conclusion of the theorem is that the metric must in fact be quasi-Einstein.

The structure of the paper is as follows. After some preliminaries in \S\ref{prel},
we give several forms for the conformally soliton equation in \S\ref{eqns}. We then
determine in \S\ref{Kvec}, in the context of the first metric
type considered, certain implications of the assumption that vector fields that occur
in the conformally soliton equation are classically distinguished. One such assumption which
does not occur in nontrivial cases has, nonetheless, an interesting classification, which
we give in an appendix in \S\ref{appen}. In \S\ref{SKR} we recall the salient features of
SKR metric theory. The main theorem in the conformally K\"ahler case is given in \S\ref{fun-dep},
and the two main theorems for special quasi-soliton metrics appear in \S\ref{qua}.

The author acknowledges the contribution of Andrzej Derdzinski to this work, most
significantly in \S\ref{fun-dep} and the appendix. The paper is dedicated to
Vanessa Gunter, whose insightful suggestion led to the results of
\S\ref{qua}.

\section{Preliminaries}\lb{prel}
\setcounter{equation}{0}
Let $(M,g)$ be a Riemannian manifold of dimension $n$, and
$\t:M\to {\mathbb{R}}\,$ a $\,C^\infty$ function. We write metrics
conformally related to $g$ in the form $\wht{g}=\t^{-2}g$, with $\t$ a smooth function.
We recall a few conformal change formulas. The covariant derivative is
\be\lb{cov}\wht{\nab}_wu=
\nab_wu-(d_w\log\t )u-(d_{u}\log\t ) w+\langle w,u\rangle\nab\log\t,\end{equation}
where $d_u$ denotes the directional derivative of a vector field $u$
and the angle brackets stand for $g$.
It follows that the $\wht{g}$-Hessian and $\wht{g}$-Laplacian of a $C^2$ function
$f$ are given by
\begin{equation}\lb{hess-lap}
\begin{array}{l}
\wht\nabla df\,=\,\nabla df\,+\,\t^{-1}[2\,d\t\odot df
-\,g(\nab\t,\nab f)g],\\
\wht{\Lap} f=\t^2\Lap f - (n-2)\t g(\nab\t,\nab f),
\end{array}
\end{equation}
where $d\t\odot df=(d\t\otimes df+df\otimes d\t)/2$.
Finally, the well-known formula for the Ricci tensor of $\wht{g}$ is given by
\be\lb{Ricci}
\wht{\ri}\,=\,\ri\,+\,(n-2)\,\t^{-1}\nabla d\t\,+\,
\left[\t^{-1}\Delta\t\,-\,(n-1)\,\t^{-2}|\nab\t|^2\right]g, \\
\end{equation}
with $\Delta$ denoting the Laplacian and the norm $|\cdot|$ is with
respect to $g$.

Recall that a (real) vector field $w$ on a complex manifold $(M,J)$
is holomorphic if the Lie derivative ${\cal{L}}_wJ=0$.
\begin{Proposition}\lb{hol}
Let $\wht{\nab}$ be a torsion-free connection on a complex manifold $(M,J)$.
For any vector field $w$, $${\cal{L}}_wJ=\wht{\nabla}_wJ+[J,\wht{\nab} w],$$
where the square brackets denote the commutator.
\end{Proposition}
In fact, write $({\cal{L}}_wJ)u={\cal{L}}_w(Ju)-J{\cal{L}}_wu$
and replace each Lie derivative by the Lie brackets, and each
of these by the torsion free condition for $\wht{\nab}$, giving
$\wht{\nab}_wJu-\wht{\nab}_{Ju}w-J\wht{\nab}_wu+J\wht{\nab}_uw$.
The first and third terms together give $(\wht{\nabla}_wJ)(u)$,
while the second and fourth terms give $[J,\wht{\nab} w](u)$.
\begin{Proposition}\lb{hess-J}
Let $(M,J)$ be a complex manifold with a Hermitian metric $\wht{g}$.
Given a $C^2$ function $f$ on $M$, set $w=\wht{\nab} f$. Then
$\wht{\nab} df$ is $J$-invariant if and only if $[J,\wht{\nab} w]=0$.
\end{Proposition}
In fact, $\wht{\nab} df (Ja,b)=\wht{g}(Ja,\wht{\nab}_bw)=
-\wht{g}(a,J\wht{\nab}_bw)=-\wht{g}(a,J(\wht{\nab}w) (b))$ while\\
$-\wht{\nab} df(a,Jb)=-\wht{g}(a,\wht{\nab}_{Jb}w)=-\wht{g}(a,(\wht{\nab}w)(Jb))$.

In the following well-known proposition $\imath_v$ denotes interior multiplication by a
vector field, while $\delta$ denotes the divergence operator.
\begin{Proposition}\lb{dY}
Let $\sig$ be a smooth function on a K\"ahler manifold such that $v=\nab\sig$
is a holomorphic gradient vector field. Then
$2\imath_v\ri =-dY$ and $2\delta\nab d\sig=dY$ for $Y=\Lap\sig$.
\end{Proposition}
For a proof, see \cite[(5.4) and (2.9)c)]{dr-ma1}.

\section{Various forms of the conformally-soliton equation}\lb{eqns}
\setcounter{equation}{0}
Let $g$ be a Riemannian metric and $\t$ a smooth function on a given manifold, for which
$\wht{g}=g/\t^2$ is a gradient Ricci soliton with soliton potential $f$.
The soliton equation for $\wht{g}$, together with its associated scalar equation, are
\be\lb{soli}
\begin{aligned}
\text{i)}&\ \ \text{$\wht{\ri}+\wht{\nab}df=\lambda \wht{g}$,\quad with $\lam$ constant.}\\[-4pt]
\text{ii)}&\ \ \text{$\wht{\Lap} f -\wht{g}(\wht{\nab} f,\wht{\nab} f)+2\lam f=a$,\quad for a constant $a$.}
\end{aligned}
\end{equation}
To obtain this in terms of $g$, we apply Equation \Ref{Ricci} and the second equation in \Ref{hess-lap} to Equation (\ref{soli}.i). The result is
\be\lb{main1}\ri+(n-2)\t^{-1}\nab d\t+ \nab df+2\t^{-1}\,d\t\odot df=\gamma g.\end{equation}
for \be\lb{gamma}\gamma=\t^{-2}[\lam+(n-1)|\nab\t|^2]-\t^{-1}[\Delta\t-g(\nab\t,\nab f)],\end{equation}
with $|\nab\t|^2=g(\nab\t,\nab\t)$.

We will now rewrite Equation \Ref{main1} in a different form.
Specifically, for the vector fields $v=\nab\t$ and
$w=\t ^2\nab f$, Equation \Ref{main1} is equivalent to
\be\lb{main-Lie}\mathrm{Ric}+\al\cal{L}_vg+\bet\cal{L}_wg=\gamma g,\end{equation}
with $\al=(n-2)\t^{-1}/2$, $\bet=(2\t^2)^{-1}$, and $\cal{L}$ denoting the Lie derivative.
To show this, recall that for any vector fields $a$,$b$
\be\lb{Lg}(\cal{L}_w g)(a,b)=g(\nab_a w,b)+g(a,\nab_b w),\end{equation} or
$\cal{L}_wg=[\nab w+(\nab w)^*]_\flat$,
where $*$ denotes the adjoint and $\flat$ is the isomorphism associated with lowering of an index.
Now clearly
$\cal{L}_vg=\cal{L}_{\nab\t}g=2\nab d\t$.
To compute the Lie derivative term for $w$, write $w=h\nab f$, then
\begin{multline}
\cal{L}_w g=[\nab (h\nab f)+(\nab (h\nab f))^*]_\flat
=h\nab df + dh\otimes df + h\nab df + df\otimes dh\\\nonumber
=2h\nab df+2\,dh\odot df.
\end{multline}
Setting $h=\t^2$ and dividing
by $2\t^{2}$ gives
$\nab df + 2\t^{-1}d\t\odot df=(2\t^2)^{-1}\cal{L}_{\t^2\nab f}g
=\bet\cal{L}_wg$.

Another form for Equation \Ref{main1} is obtained as follows. It is natural
to combine the two Hessian terms into one. For this, set
$$\mu=\log\t, \quad \thet =f+(n-2)\log\t,\quad\psi=2\thet-(n-2)\mu.$$
Then \Ref{main1}, \Ref{gamma} and (\ref{soli}.ii) read
\be\lb{2form}
\begin{aligned}
\mathrm{i})&\ \ \ri+\nab d\thet+d\mu\odot d\psi=\gamma g,\quad
\gamma=\lam e^{-2\mu}-\Lap\mu+g(\nab\thet,\nab\mu),\\[-4pt]
\mathrm{ii})&\ \ e^{2\mu}[\Lap f-g(\nab\thet,\nab f)]+2\lam f=a.
\end{aligned}
\end{equation}
To derive (\ref{2form}.ii) one uses the second equation in \Ref{hess-lap}, which, in terms
of $\mu$, reads $$e^{-2\mu}\wht{\Lap} f=\Lap f - (n-2)g(\nab\mu,\nab f).$$

\section{The K\"ahler condition and distinguished vector fields}\lb{Kvec}
\setcounter{equation}{0}
Let $g$ be a metric which is K\"ahler with respect to a complex structure $J$ on a manifold $M$,
and conformal to a gradient Ricci soliton. Equation \Ref{main-Lie} then holds, and the
$J$-invariance of $g$ and its Ricci curvature implies that
\be\lb{herm}\text{$\al\cal{L}_vg+\bet\cal{L}_wg$ is $J$-invariant.}\end{equation} Applying \Ref{Lg}
to the relation $\Ll_xg(J\cdot,\cdot)=-\Ll_xg(\cdot,J\cdot)$, for both $x=v$ and $x=w$, and recalling
that $J^*=-J$, we see that \Ref{herm} is equivalent to the vanishing of a commutator:
$[\al(\nab v+(\nab v)^*)+\bet(\nab w+(\nab w)^*), J]=0,$ or
\be\lb{LJ}[\al(\cal{L}_vg)^\sharp+\bet(\cal{L}_wg)^\sharp, J]=0,\end{equation}
where $\sharp$ denotes the isomorphism acting by the raising of an index.

The most obvious case where \Ref{LJ} holds is when both summands vanish separately, so that,
$w$, for example, satisfies
\be\lb{LJ1}[(\cal{L}_wg)^\sharp, J]=0.\end{equation}
We wish to study relations between these two vanishing conditions for $v$ and $w$.
We first note that \Ref{LJ1} includes as special cases the following three classical types of vector fields
(the first being, of course, a special case of the second):
\begin{itemize}
\setlength\itemsep{-1.1em}
\item a Killing vector field ($\cal{L}_wg=0$), \\
\item a conformal vector field ($(\cal{L}_wg)^\sharp=hI$, for a function $h$ and $I$ the identity),\\
\item a holomorphic vector field ($[\nabla w,J]=0$ on a K\"ahler manifold).
\end{itemize}
This last type is holomorphic by Proposition \ref{hol} in the K\"ahler case, and
it is indeed a special case, as $[\nabla w,J]^*=[(\nabla w)^*, J]$ and
\Ref{LJ1} is equivalent to $[\nab w+(\nab w)^*,J]=0$.

We will see in the next theorem that the Killing case does not lead to important
K\"ahler conformally soliton metrics. However, K\"ahler metrics with a Killing field of the form
$w=\t^2\nab f$  can be classified, as we show in the appendix.


To state the next result, we continue to assume $g$ is K\"ahler
and conformal to a gradient Ricci soliton $\wht{g}$, but now on a manifold of dimension $n>2$.
With notations as above for $\t$, $f$, $v$ and $w$ we have
\begin{Theorem}\lb{vec-fl} The following conclusions follow for the vector fields $v$ and $w$:\\
\begin{minipage}[t]{\linegoal}
\begin{enumerate}[leftmargin=*]
\setlength\itemsep{-1.1em}
\item If $w$ is a Killing or, more generally, a conformal vector field for $g$, then $\wht{g}$ is Einstein.\\
\item If $w$ is a holomorphic vector field and either $v$ is holomorphic as well,
or $\wht{\nab} df$ is $J$-invariant, then $\mathrm{span}_\mathbb{C}\{v\}=\mathrm{span}_\mathbb{C}\{w\}$ away from the
zero sets of $v$ and $w$.
\end{enumerate}
\end{minipage}
\end{Theorem}

\begin{proof}
The key to both parts is that $w=\t^2\nab f$ is also the $\wht{g}$-gradient of $f$, i.e.
$w=\wht{\nab} f$. Therefore $\Ll_wg=\Ll_{\wht{\nab}f} g=2\wht{\nab} df$. It follows that in
the case where $w$ is Killing, or, more generally, conformal, the Ricci soliton equation in \Ref{soli}
reduces, using Schur's lemma, to the Einstein equation. This proves (1).

To prove (2), note first that the combination of Propositions \ref{hol} and \ref{hess-J}
for a K\"ahler metric yields the result that the vector field $v=\nab\t$ is holomorphic
exactly when $\nab d\t$ is $J$-invariant. This in turn is equivalent, by \Ref{Ricci} and the fact that
the metric $g$ and its Ricci curvature are $J$-invariant, to $\wht{Ric}$ being $J$-invariant. Finally, the
latter condition is equivalent to $\wht{\nab}df$ being $J$-invariant, by the soliton equation in \Ref{soli}.
The combination, again, of Propositions \ref{hol} and \ref{hess-J}, but this time for a hermitian metric, yields
equivalence of the latter condition with $\cal{L}_{\wht{\nab}f}J=\wht{\nab}_{\wht{\nab}f}J$, or
\be\lb{wJ}\cal{L}_{w}J=\wht{\nab}_{w}J.\end{equation}

Now from \Ref{cov}, for any vector field $u$
\begin{eqnarray*}(\wht{\nab}_wJ)u&=&\wht{\nab}_w(Ju)-J\wht{\nab}_wu=
\nab_w(Ju)-\t^{-1}(d_w\t )Ju-\t^{-1}(d_{Ju}\t ) w+\langle w,Ju\rangle\t^{-1}v\\
&-&[J\nab_w(u)-\t^{-1}(d_w\t )Ju-\t^{-1}(d_{u}\t )Jw+\langle w,u\rangle\t^{-1}Jv]\\
&=&\t^{-1}\left(-\langle v,Ju \rangle w+
\langle w,Ju\rangle v+\langle v,u \rangle Jw-\langle w,u\rangle Jv\right)\\
&=&\t^{-1}\left(\langle Jv,u \rangle w-
\langle Jw,u\rangle v+\langle v,u \rangle Jw-\langle w,u\rangle Jv\right)\\
\end{eqnarray*}
where we used the fact that $\nab_wJ=0$, and the angle brackets denote $g$.
Combining this with \Ref{wJ} we see that as $w$ is holomorphic, the last
expression vanishes for {\em every} vector field $u$. Substituting first $u=v$ and
then $u=Jv$ shows that away from the zeros of $v$, the vector fields $w$ and $Jw$ are
in $\mathrm{span} \{v,Jv\}$. As this reasoning is symmetric for $v$ and $w$, the result follows.


\end{proof}

In known examples the manifolds on which $g$ and $\wht{g}$ live are locally total spaces
of holomorphic line bundles over manifolds admitting a K\"ahler-Einstein metric, and
there in fact $\mathrm{span}_\mathbb{R}v=\mathrm{span}_\mathbb{R}w$.

\section{SKR metrics}\lb{SKR}

We recall here some facts from \cite{dr-ma1} and \cite{confsol} on the notion of an SKR
metric, i.e. a K\"ahler metric $g$ admitting a special K\"ahler-Ricci potential $\sig$.
For the definition, recall that a smooth function $\sig$ on a K\"ahler manifold $(M,J,g)$ is called a Killing
potential if $J\nab\t$ is a Killing vector field. The definition of a special K\"ahler-Ricci potential
consists then of the requirement that $\sig$ is a Killing potential and, at each noncritical point of it, all nonzero
tangent vectors orthogonal to the complex span of $\nab\sig$ are eigenvectors of both the
Ricci tensor and the Hessian of $\sig$, considered as operators. This rather technical definition
implies that a Ricci-Hessian equation holds for $\sig$ on a suitable open set (see \cite[Remark 7.4]{dr-ma1}),
namely \be\lb{Ri-He}
\ri+\al\nab d\sig=\gamma g,
\end{equation}
for some functions $\al$, $\gamma$ which are functionally dependent on $\sig$.

We say that Equation \Ref{Ri-He} is a {\em standard} Ricci-Hessian equation if $\al\, d\al\ne 0$
whenever $d\sig\ne 0$. This condition will appear in all our main theorems. However,
even if it does not hold over the entire set where $d\sig\ne 0$, these theorems will hold,
with the same proofs, on any open subset of $\{d\sig\ne0 \}$ where $\al d\al\ne 0$. We have
\begin{Proposition}\lb{typ}
A K\"ahler metric on a manifold of dimension at least four is an SKR metric,
provided it satisfies a standard Ricci-Hessian Equation of the form \Ref{Ri-He} with
$d\al\we d\sig=d\gamma\we d\sig=0$.
\end{Proposition}
This result appears in \cite[Proposition 3.5]{confsol} with proof referenced from \cite{dr-ma1}, a proof
that has to be interpreted with the aid of \cite[Remark 3.6]{confsol}. Note also that in dimension
greater than four, if the  Ricci-Hessian equation of a K\"ahler metric satisfies $d\al\we d\sig=0$
then it automatically also satisfies $d\gamma\we d\sig=0$ (see \cite[Proposition 3.3]{confsol}).

If an SKR metric is locally irreducible, the theory of such metrics
(see \S$4$ of \cite{confsol}) implies that a pair of equations holds on
the open set where the Ricci-Hessian equation \Ref{Ri-He} holds:
\be\lb{ODEs}\begin{aligned}
&(\sig-c)^2\phi''+(\sig-c)[m-(\sig-c)\al]\phi'-m\phi=K.\\
&-(\sig-c)\phi''+[\al(\sig-c)-(m+1)]\phi'+\al\phi=\gamma
\end{aligned}
\end{equation}
Here $\phi$ is defined pointwise as the eigenvalue of the Hessian of $\sig$,
considered as an operator, corresponding to the eigendistribution
$[\mathrm{span}_\mathbb{C}\nab\sig]^\perp$, and $c$ is a constant. This
eigenvalue and $\sig$ are functionally dependent, so that the primes
represent differentiations with respect to $\sig$. Furthermore, $K$ is
a constant whose exact expression in terms of SKR data will not concern
us, while $m=\dim (M)/2$. We further have the following relations between
$\phi$, $\Lap\sig$ and $Q:=g(\nab\sig,\nab\sig)$:
\be\lb{YQ}
\Lap\sig=2m\phi+2(\t-c)\phi',\quad Q=2(\t-c)\phi.
\end{equation}

In analyzing equations such as \Ref{ODEs} we will repeatedly use in \S\ref{qua} the following elementary lemma,
taken from \cite{confsol}.
\begin{Lemma}\lb{DElem}
For a system
\begin{eqnarray}
A\phi''+B\phi'+C\phi&=&D\\
\phi'+p\phi&=&q\nonumber
\end{eqnarray}
with rational coefficients,
either
$A(p^2-p')-Bp+C=0$ holds identically,
or else the solution is given by
$\phi=(D-A(q'-pq)-Bq)/(A(p^2-p')-Bp+C)$.
\end{Lemma}

Metrics with a special K\"ahler-Ricci potential have been completely classified
\cite{dr-ma1, skrp}. One result that will be used below is that for an irreducible SKR metric,
the function $\phi$ is nowhere zero on the open dense set where $d\sig\ne 0$.

\section{Functional dependence}\lb{fun-dep}
\setcounter{equation}{0}
Recall Equation (\ref{2form}.i):
\be\lb{2form2}
\ri+\nab d\thet+d\mu\odot d\psi=\gamma g,\quad
\gamma=\lam e^{-2\mu}-\Lap\mu+g(\nab\thet,\nab\mu),
\end{equation}
with $\mu=\log\t$, $\thet =f+(n-2)\log\t$ and $\psi=2\thet-(n-2)\mu.$
This was one of the forms of Equation \Ref{main1} characterizing
a metric $g$ conformal to a gradient Ricci soliton. If $g$ is also
K\"ahler on a manifold $(M,J)$ of real dimension at least four,
constancy of $\theta$ implies that $g$ is in fact K\"ahler-Einstein.
This follows since, in this case, the above relation defining $\psi$
shows that the term $d\mu\odot d\psi$ is just a constant multiple
of $d\mu\otimes d\mu$, and the latter vanishes as it is the only
term in \Ref{2form2} that is not $J$-invariant.

Note that $f$ cannot be constant on a nonempty open subset of $M$ without being
constant everywhere in $M$, by a real-analyticity argument stemming
from \cite{ivey1}. Hence the same holds for $\thet$, because we see from the previous paragraph
that constancy of $\thet$ on a nonempty open set implies the same for $f$.
\begin{Proposition}\lb{R-H}
Assume $g$ is K\"ahler and conformal to a gradient Ricci soliton in dimension $n\geq 4$
with $\thet$ nonconstant. If $$df\we d\t= 0$$ (equivalently, $d\mu\we d\thet=0$)
then $g$ satisfies on an open dense set a Ricci-Hessian equation of the form
\be\lb{RcHs}\al\nab d\sig+\ri=\gamma g,\end{equation}
for appropriate functionally dependent functions $\al$, $\sig$.
\end{Proposition}

In fact, in the set where $d\theta \ne 0$, choose any
function $t$ of $\theta$ with $dt \ne 0$, so that $\theta$ and $\mu$
become functions of $t$, on some interval of the variable $t$.
For the moment, $t$ is not further specified. Denoting $()^{\dot{}} =\frac d{dt}$, we have
\be\lb{tvar}\nab d\thet+d\mu\odot d\psi=\dot{\thet}\nab dt+[\ddot{\thet}+
2\dot{\mu}\dot{\thet}-(n-2)\dot{\mu}^2]dt\odot dt.\end{equation}
Next, we choose a function $\sig$ of $t$ such that $\dot{\sig}>0$ and
\be\lb{genvar}\ddot{\sig}/\dot{\sig}=[\ddot{\thet}+2\dot{\mu}\dot{\thet}
-(n-2)\dot{\mu}^2]/\dot{\thet}\end{equation}
on the open dense set where $\dot{\thet}\ne 0$. The right hand side of this
equation is given, so that this stipulation amounts to the requirement that
an (easily solvable) ODE holds for $\sig$, with an essentially unique solution.

We now fix $t=\sig$. For this choice, \Ref{genvar} becomes
\be\lb{ode}\ddot{\thet}+2\dot{\mu}\dot{\thet}-(n-2)\dot{\mu}^2=0,\end{equation}
which holds on the image under $\sig$ of an open dense set, namely the
intersection of the noncritical set of $\sig$, with points where
$\dot{\thet}\ne 0$. It follows from \Ref{ode} and \Ref{tvar} that $\nab d\thet+d\mu\odot d\psi=\al\nab d\sig$, with $\al=\dot{\thet}$. This translates the first of Equations \Ref{2form2} into a Ricci-Hessian equation.

We now record some relations that will be used in the next theorem, with assumptions as in
Proposition \ref{R-H}. Let $Q=g(\nab\sig,\nab\sig)$, $Y=\Lap\sig$ and $s$ the scalar curvature of $g$.
First, from \Ref{2form2},
\be\lb{gam}\gamma =\lam e^{-2\mu}-\dot{\mu}Y+(\al\dot{\mu}-\ddot{\mu})Q\end{equation}
as $\Lap\mu=\dot{\mu}Y+\ddot{\mu}Q$ and $g(\nab\thet,\nab\mu)=\al\dot{\mu}Q$.
Next, we have
\be\lb{rels}
\begin{aligned}
\mathrm{i})&\ \al Y+s=n\gamma,\ \ \ \  \mathrm{ii})\ \al\, dY+Y\dot{\al}\,d\sig+ds=nd\gamma,\\[-4pt]
\mathrm{iii})&\ \al\, dY+\dot{\al}\, dQ+ds=2\,d\gamma,\ \ \ \ \ \mathrm{iv})\ \al\, dQ-dY=2\gamma\, d\sig.
\end{aligned}
\end{equation}
These equations are obtained in succession by
taking the $g$-trace of \Ref{RcHs}; forming the $d$-image of (\ref{rels}.i);
finally, applying twice the divergence operator $2\delta$ and, separately, interior multiplication
by $\nab\sig$, i.e. $2\imath_{\nab\sig}$,
to \Ref{RcHs} and using Proposition \ref{dY} and the Bianchi relation $2\delta\ri=ds$.

Further relations are obtained by subtracting (\ref{rels}.iii) from (\ref{rels}.ii), then applying
$\ldots \we d\sig$ to (\ref{wedge}.a), $d$ to (\ref{rels}.iv) and $d$ followed by $\ldots \we d\sig$ to
\Ref{gam}, which yield
\be\lb{wedge}
\begin{aligned}
\mathrm{a})&\ Y\dot{\al}\, d\sig - \dot{\al}\, dQ = (n-2)\, d\gamma ,\\[-4pt]
\mathrm{b})&\ \dot{\al}\, d\sig \we dQ = (n-2)\, d\gamma \we d\sig ,\\[-4pt]
\mathrm{c})&\ \dot{\al}\, d\sig \we dQ = 2\, d\gamma \we d\sig ,\\[-4pt]
\mathrm{d})&\ d\gamma \we d\sig = (\al\dot{\mu} - \ddot{\mu})\, dQ \we d\sig - \dot{\mu}\, dY \we d\sig .
\end{aligned}
\end{equation}
We can now state the following partial classification theorem.
\begin{Theorem}\lb{classf}
Let $g$ be a K\"ahler metric conformal to a gradient Ricci soliton $\wht{g}$ on a manifold $M$
of dimension $n\geq 4$, so that Equations \Ref{main1} and \Ref{2form2} hold.
If $df\we d\t= 0$ (equivalently, $d\mu\we d\thet=0$), then one of the following must occur:
\be\begin{aligned}
(\mathrm{i)}&\ \  \text{$g$ is a K\"ahler-Ricci soliton,}\\[-4pt]
(\mathrm{ii)}&\ \  \text{$g$ satisfies a Ricci-Hessian equation, and if it is standard, $g$ is an SKR metric,}\\[-4pt]
(\mathrm{iii)}&\ \  \text{$n=4$  and $\wht{g}$ is an Einstein metric,}\\[-4pt]
(\mathrm{iv)}&\ \  \text{$n=4$ and $\wht{g}$ is a non-Einstein steady gradient Ricci soliton ($\lambda=0$).}
\end{aligned}\end{equation}
The Ricci-Hessian equation in (ii) holds on an open dense set.
\end{Theorem}
Note that the less expected possibility here is (iv). However, the theorem shows
it cannot occur when $M$ is compact, as it is well-known that
compact manifolds do not admit non-Einstein steady gradient Ricci solitons (see \cite{ivey}).
\begin{proof}
If $\theta$ is constant, we have seen $g$ is K\"ahler-Einstein, a special case of (i). Assume from
now on that $\theta$ is nonconstant. Then by Proposition \ref{R-H}, $g$ satisfies the Ricci-Hessian
equation \Ref{RcHs} on an open dense set.

When $\al$ is constant, so is $\gamma$, by (\ref{wedge}.a) and thus
\Ref{RcHs} gives (i). Next, we assume in the rest of this proof that $\al$ is nonconstant.

If $n>4$ (or, $dQ\we d\sig=0$ everywhere), then $d\gamma\we d\sig=0$, as verified by
subtracting (\ref{wedge}.c) from (\ref{wedge}.b) (or,using (\ref{wedge}.c)).
If the Ricci-Hessian equation is standard, taking to consideration that $d\al\we d\sig=0$
because $\al=\dot\theta$, Proposition \ref{typ} implies (ii).

So assume $n=4$ and $dQ\we d\sig\ne 0$ somewhere in $M$ (and, consequently, almost everywhere, by
an argument involving real-analyticity, valid in dimension four).
By (\ref{rels}.iv), (\ref{wedge}.c) and (\ref{wedge}.d),
$(\dot{\al}+2\al\dot{\mu}-2\ddot{\mu})\, dQ-2\dot{\mu}dY$ and $2\dot{\mu}(dY-\al\,dQ)$ are
both functional multiples of $d\sig$. Adding these two relations, we obtain
$(\dot{\al}-2\ddot{\mu})\,dQ\we d\sig=0$, so that \Ref{ode} with $n=4$ gives
$\dot{\al}=2\ddot{\mu}$ and
\be\lb{de-stuf}
a)\ \al=2(\dot{\mu}+p), \quad b)\ 2\dot{\al}+\al^2=4p^2, \quad c)\ 4(\al\dot{\mu}-\ddot{\mu})
=(3\al+2p)(\al-2p),
\end{equation}
for a constant $p$, where a) is obtained by integration, b) using a) and \Ref{ode} with $n=4$,
while c) follows from  a) and b) by algebraic manipulations that use again $\dot{\al}=2\ddot{\mu}$.
Also, as $\dot{\thet}=\al$,
\be\lb{const}
\mathrm{i})\ \dot{f} =2p, \quad
\mathrm{ii})\ \text{$p\,[e^{2\mu}(Y-\al Q)+2\lambda\sig]$ is a constant.}
\end{equation}
In fact, differentiating the relation $\thet = f+(n-2)\mu$ with $n=4$ and
(\ref{de-stuf}.a) give (\ref{const}.i). Thus, $f$ equals
$2p\,\sig$ plus a constant. Hence $\Lap f=2pY$, and (\ref{const}.ii) follows from (\ref{2form}.ii)
and (\ref{de-stuf}.a). If $p=0$ then $f$ is constant, and this, by
the soliton equation (first equation in \Ref{soli}), implies (iii).

Suppose, finally, that $p\ne 0$ while $n=4$ and $dQ\we d\sig\ne 0$ somewhere. As a consequence of
(\ref{wedge}.a) and (\ref{de-stuf}.b)
\be\lb{dgam}
4\,d\gamma =(4p^2-\al^2)(Y\,d\sig-dQ).
\end{equation}
On the other hand, \Ref{gam}, (\ref{de-stuf}.a) and (\ref{de-stuf}.c) give
\be\lb{4gam}
4\gamma=4\lam e^{-2\mu}+(\al-2p)[(\al+2p)Q+2(\al Q-Y)].
\end{equation}
Since $p\ne 0$, (\ref{const}.ii) yields $\al Q-Y=e^{-2\mu}(2\lambda\sig-b)$ for some
constant $b$, so that \Ref{4gam} and \Ref{dgam} become
\be\lb{gdg}\begin{aligned}
\mathrm{a})&\ \ 4\gamma =e^{-2\mu}[4\lam+(2\lam\sig-b)(2\al-4p)]+(\al^2-4p^2)Q.\\[-4pt]
\mathrm{b})&\ \ 4d\gamma=(4p^2-\al^2)[\al Q\,d\sig-e^{-2\mu}(2\lam\sig-b)\,d\sig-dQ].
\end{aligned}
\end{equation}
Thus $(4p^2-\al^2)[\al Q\,d\sig-e^{-2\mu}(2\lam\sig-b)\,d\sig]$ equals the sum of
$Q\,d(\al^2-4p^2)$ and $d[e^{-2\mu}\left(4\lam+(2\lam\sig-b)(2\al-4p)\right)]$,
since both expressions coincide with $4d\gamma+(4p^2-\al^2)\,dQ$, which for the
former is clear from (\ref{gdg}.b), and for the
latter follows if one applies $d$ to (\ref{gdg}.a). This equation yields
$4e^{-2\mu}(2\lam\sig-b)(2p-\al)\al=0$, as seen by evaluating these expressions
via the first two parts of \Ref{de-stuf}, and subtracting the former expression
from the latter. As we are assuming $\al$ is not constant, it follows necessarily
that $\lam$ (and $b$) must be zero. This gives (iv), completing the proof.
\end{proof}

\section{Quasi-solitons}\lb{qua}
\setcounter{equation}{0}
Many of the original examples of gradient Ricci solitons arise as
warped products over a one dimensional base (cf. \cite{many}). We consider
here the case of an arbitrary base.

Let $\overline{g}$ be a warped product (gradient Ricci) soliton metric
on a manifold $M=B\times F$, so that \be\lb{war-sol}\overline{g}=g_B+\ell^2g_F:=g+\ell^2g_F,
\quad \overline{\ri}+\overline{\nab}df=\lam\overline{g},\end{equation} where $\ell$ is the
(pullback of) a function on the base $B$ and $\lambda$ is constant.
When $\bar{g}$ is Einstein, the base metric $g=g_B$ is sometimes called
quasi-Einstein. Similarly, in our case we will call $g_B$ a {\em quasi-soliton}
metric and drop the subscript $B$ in the notation for $g_B$-dependent quantities.

\begin{Proposition}
With notations as above, the soliton equation for $\overline{g}$ (see \Ref{war-sol})
is equivalent to the system
\be\lb{quas}
\begin{aligned}
&\ri-\frac k\ell\nab d\ell+\nab df=\lam g,\quad k=\dim (F),\\[-4pt]
&\ri_F=\nu g_F,\ \ \mathrm{where}\ \nu\ \mathrm{is\ given\ by}\\[-4pt]
&\nu+\ell\,d_{\nab f}\ell-\ell^2\ell^\#=\lam \ell^2,\
\mathrm{for}\  \ell^\#=\ell^{-1}\Delta\ell+(k-1)\ell^{-2}|\nab\ell|^2.
\end{aligned}
\end{equation}
\end{Proposition}
In particular the fiber metric is Einstein if $\dim (F)>2$, and $f$ turns out to be a function
with vanishing fiber covariant derivative (see below), so that we regard it as a function on $B$.
Unlike the quasi-Einstein case, the third scalar equation in \Ref{quas} does not appear to be a
consequence of the first.
\begin{proof}
To derive the equations, we need the well-known Ricci curvature
formulas for warped products (see \cite{onl}), and additionally, similar equations for the Hessian of
$f$. For the latter we use the covariant derivative formulas for warped products, together with
the known fact that for a $C^1$ function defined on the base, the gradient of its pull-back equals the
pull-back of its base gradient.

Let, $x$, $y$ denote lifts of vector fields on $B$, and $u$, $v$ lifts of vector fields on
$F$. Then we have
\be
\begin{aligned}
&\text{$\ol{\nab}_xy$ is the lift of $\nab_xy$ on $B$,}\\[-4pt]
&\text{$\ol{\nab}_xv=\ol{\nab}_vx=d_x\log (\ell) v$,}\\[-4pt]
&\text{$[\ol{\nab}_vw]^F$ is the lift of $\nab^F_vw$ on $F$,}\\[-4pt]
&\text{$[\ol{\nab}_vw]^B=-\ol{g}(v,w)\nab\log(\ell).$}
\end{aligned}
\end{equation}
Hence,
\be\begin{aligned}
\ol{\nab} df(x,y)&=\ol{g}(\ol{\nab}_x\ol{\nab} f,y)
=\ol{g}(\ol{\nab}_x(\ol{\nab} f)^B,y)+\ol{g}(\ol{\nab}_x(\ol{\nab} f)^F,y)=\\[-4pt]
&g({\nab}_x(\ol{\nab} f)^B,y)+\ol{g}(d_x\log (\ell)(\ol{\nab} f)^F,y)=g({\nab}_x(\ol{\nab} f)^B,y),\\[-4pt]
\ol{\nab} df(x,v)&=\ol{g}(\ol{\nab}_x\ol{\nab} f,v)
=\ol{g}(\ol{\nab}_x(\ol{\nab} f)^B,v)+\ol{g}(\ol{\nab}_x(\ol{\nab} f)^F,v)=\\[-4pt]
&d_x\log (\ell)\ol{g}((\ol{\nab} f)^F,v)=\ell d_x\ell g_F((\ol{\nab}f)^F,v),\\[-4pt]
\ol{\nab} df(v,w)&=\ol{g}(\ol{\nab}_v\ol{\nab} f,w)
=\ol{g}(\ol{\nab}_v(\ol{\nab} f)^B,w)+\ol{g}(\ol{\nab}_v(\ol{\nab} f)^F,w)=\\[-4pt]
&d_{(\ol{\nab} f)^B}(\log\ell)\ol{g}(v,w)+\ol{g}(\nab^F_v(\ol{\nab} f)^F,w)
-\ol{g}(v,(\ol{\nab} f)^F)\ol{g}(\nab\log(\ell), w)=\\[-4pt]
&\ell d_{(\ol{\nab} f)^B}(\ell)g_F(v,w)+\ell^2g_F(\nab^F_v(\ol{\nab} f)^F,w).
\end{aligned}\end{equation}
We combine these with the Ricci curvature formulas
\be\begin{aligned}
&\ri (x,y)=\ri_B(x,y)-(k/\ell)\nab d\ell(x,y),\\[-4pt]
&\ri (x,v)=0,\\[-4pt]
&\ri (v,w)=\ri_F(v,w)-\ell^\#g(v,w).
\end{aligned}\end{equation}
We now notice that the soliton equation applied to $x$ and $v$ implies that
$(\ol{\nab} f)^F=0$ so that $f$ can be regarded as the pull-back of a function on $B$.
This readily gives Equations \Ref{quas}.
\end{proof}

In analogy with the previous section, we will be considering quasi-soliton metrics
for which $f$ and $\ell$ are functionally dependent, that is $$df\we d\ell=0.$$
We call such metrics {\em special quasi-soliton} metrics.

It is known that K\"ahler quasi-Einstein metrics do not exist on a compact manifold,
and in general must be certain Riemannian product metrics \cite{quasE}. Similarly we show
\begin{Theorem}\lb{quas1}
Let $g$ be a K\"ahler special quasi-soliton metric on a manifold $M$ of dimension at least four. Then
$g$ satisfies a Ricci-Hessian equation on an open set. If this equation is standard, then
$g$ is a Riemannian product there. If the dimension is greater than four, then one of the factors
in this product is a K\"ahler-Einstein manifold of codimension two.
\end{Theorem}
\begin{proof}
As the quasi-soliton metric is special, we have $\nab df=f'\nab d\ell+f''d\ell\otimes d\ell$, with
the prime denotes differentiation with respect to $\ell$. The first of Equations \Ref{quas} then
becomes \be\lb{quas-sp}\ri+(f'-\frac k\ell)\nab d\ell+f''d\ell\otimes d\ell=\lambda g.\end{equation}

In analogy with Proposition \ref{R-H}, we introduce
a function $\sig$ with $d\ell\we d\sig=0$ and rewrite the special quasi-Einstein Equation
\Ref{quas-sp} as
\be\lb{sigm}\ri+\ti{\al}\ell'\nab d\sig+(\ti{\al}\ell''+f''\ell'^2)d\sig\otimes d\sig=\lam g,\end{equation}
for $\ti{\al}=f'(\ell)-k/\ell$, with the convention that primes on $\ell$
represent differentiations with respect to $\sig$,
while primes on $f$ still represent differentiations with respect to $\ell$.
The restriction on the open set where an ODE analogous to \Ref{genvar} holds
is $\al:=\ti{\al}\ell'\ne 0$ (corresponding to $\dot{\theta}\ne 0$ in Proposition \ref{R-H}).
On that set, Equation \Ref{sigm} becomes a Ricci-Hessian equation of the form
$$\ri+\al\nab d\sig=\lambda g,\quad \al=\ti{\al} \ell',$$ provided we choose $\sig$ so that the
differential equation
 \be\lb{aux}\ti{\al}\ell''+f''\ell'^2=0\end{equation} also holds.

Assuming the Ricci-Hessian equation is standard, Proposition \ref{typ} now shows that $g$ is an SKR metric
on the open set described above.
If $g$ is irreducible, the theory of SKR metrics gives the two equations \Ref{ODEs}, which now
take the form
\be\lb{ODEs1}\begin{aligned}
&(\sig-c)^2\phi''+(\sig-c)[m-(\sig-c)\al]\phi'-m\phi=K,\\
&-(\sig-c)\phi''+[\al(\sig-c)-(m+1)]\phi'+\al\phi=\lam,
\end{aligned}
\end{equation}
where $\phi$ is defined pointwise as the eigenvalue of the Hessian of $\sig$,
mentioned in \S\ref{SKR}.

Adding the first of Equations \Ref{ODEs1} to $(\sig-c)$ times the second replaces the latter with the
first order equation $$-(\sig-c)\phi'+[(\sig-c)\al-m]\phi=K+(\sig-c)\lam.$$
Denote the ratio of the second coefficient of this equation to the first by $p$,
the ratio of the third to the first by $q$ and the coefficients of the
first of Equations \Ref{ODEs1} by $A$, $B$, $C$, $D$. We wish to invoke Lemma \ref{DElem}.
An easy computation gives the two relations
\be\lb{quant}
\begin{aligned}
&A(p^2-p')-Bp+C= \al'(\sig-c)^2,\\
&D-A(q'-pq)-Bq=0
\end{aligned}\end{equation}
According to this lemma, the solution $\phi$ is the ratio of
the second term to the first, if the latter is nonzero. However,
as mentioned at the end of \S\ref{SKR}, the function $\phi$ is
nowhere zero on the set where $d\sig\ne 0$ when $g$ is irreducible.
Hence the only possibility is that the first term in \Ref{quant} vanishes
identically, i.e. $\al$ is constant, so that $g$ is additionally a gradient
Ricci soliton. Writing this condition explicitly we get, with primes
now denoting solely differentiations with respect to $\sig$,
$$(f\circ\ell)'-k\ell'/\ell=b$$ where $b$ is constant.
But Equation \Ref{aux} can also be written as $$(f\circ\ell)''
-k\ell''/\ell=0.$$ Differentiating the first
of these two equations and combining it with the second shows that
$\ell$ is constant, hence $g$ is Einstein. But this means $\al\equiv 0$,
contradicting that the Ricci-Hessian equation for $g$ is standard.
Hence $g$ must be reducible.
The structure of the Riemannian product constituting $g$ follows from SKR theory.
\end{proof}

Next we consider the problem of whether quasi-soliton metrics can be conformally K\"ahler. This is
certainly possible for quasi-Einstein metrics (see \cite{quEin,quaE1}). We have the following result, analogous in form and
in proof to the previous one, though it requires more assumptions and is computationally more difficult.
\begin{Theorem}\lb{quas2}
Let $M$ be a manifold of dimension $n=2m>4$ and $g$ an irreducible K\"ahler metric on $M$ conformal to a
special quasi-soliton $\wht{g}=g/\t^2$ having warping function $\ell$, potential $f$ and appropriate constants
$k$ and $\lambda$. Assume $\t$ is a Killing potential for $g$ and $d\ell\we d\t=0$.
Then $g$ satisfies a Ricci-Hessian equation. If the latter is standard, then $\wht{g}$ is quasi-Einstein.
\end{Theorem}
\begin{proof}
Being a special quasi-soliton, $\wht{g}$ satisfies Equation \Ref{quas-sp}, i.e.
\begin{equation}\lb{QRH}
\wht{\ri}+\mu\wht{\nabla} d\ell+\chi d\ell\otimes d\ell=\lambda\wht{g},
\end{equation}
for $\mu=f'(\ell)-k/\ell$ and $\chi=f''(\ell)$.

 Using \Ref{Ricci} and the first equation in \Ref{hess-lap}, we see that $g$ satisfies
\begin{multline}\ri+(n-2)\t^{-1}\nab d\t+(\t^{-1}\Delta\t-(n-1)\t^{-2}Q)g+\\
\mu (\nab d\ell+2\t^{-1}d\t\odot d\ell-\t^{-1}g(\nab\t,\nab \ell)g)+\chi d\ell\otimes d\ell=\lambda\t^{-2}g,
\end{multline}
with $Q=g(\nab\t,\nab\t)$. Since $d\ell\we d\t=0$, writing $d\ell=\ell'(\t)\,d\t$ and rearranging terms, we
rewrite this equation as
\begin{multline}\lb{confK}
\ri+\al\nab d\t+
(\mu (\ell''+2\t^{-1}\ell')+(\ell')^2\chi )d\t\otimes d\t\\
=(\lambda\t^{-2}-\t^{-1}\!\Delta\t+(\al+\t^{-1}\!)\t^{-1}\!Q)g,\quad \mathrm{for}\ \ \al=(n-2)\t^{-1}+\mu\ell'.\\
\end{multline}

\vspace{-.1in}
As $g$ is K\"ahler and $\t$ is a Killing potential, the term with
$d\t\otimes d\t$ is the only one which is not $J$-invariant.
Hence its coefficient must vanish:
\be\lb{coef}\mu (\ell''+2\t^{-1}\ell')+(\ell')^2\chi =0.\end{equation}
As a result, Equation \Ref{confK} is Ricci-Hessian:
\be\lb{RH}
\text{$\ri+\al\nab d\t=\gamma g$,\quad  where
$\gamma=\lambda\t^{-2}-\t^{-1}\!\Delta\t+(\al+\t^{-1}\!)\t^{-1}\!Q$.}
\end{equation}
Since clearly $d\al\we d\t=0$, and $n>4$, as mentioned in \S\ref{SKR},
we also have $d\gamma\we d\t=0$. Under the assumption that the
Ricci-Hessian equation is standard, we conclude from Proposition \ref{typ}
that $(g,\t)$ is an SKR metric with $\t$ the special K\"ahler-Ricci potential.
As in the previous theorem, irreducibility of $g$ again implies that
two ODE's hold for the horizontal Hessian eigenvalue function $\phi$. They are
\be\lb{odes}
\begin{aligned}
&(\t-c)^2\phi''+(\t-c)[m-(\t-c)\al]\phi'-m\phi=K\\
&-(\t-c)\phi''+(\al(\t-c)-(m+1))\phi'+\al\phi=\gamma=\\
&\lambda\t^{-2}-\t^{-1}(2m\phi+2(\t-c)\phi')+(\al+\t^{-1})\t^{-1}2(\t-c)\phi
\end{aligned}
\end{equation}
where $K$, $c$ are constants, and we have used formulas \Ref{YQ} giving $\Lap\t$ and $Q$
in terms of $\phi$.

Simplifying the second equation, we then replace it by a first order equation as in the previous
theorem, to obtain the equivalent system
\begin{multline}\lb{odes-fin}
(\t-c)^2\phi''+(\t-c)[m-(\t-c)\al]\phi'-m\phi=K,\\
{\frac { \left( \t-c \right)  \left( \t-2c \right)}
{\t}}\phi'- \left( {\frac { \left( \t-c \right)
 \left( \t-2c \right)  }{\t}}\al+{\frac {2(\t-c)^2-m\t(\t-2c)}{{\t}^{2}}} \right) \phi \\
 ={\frac {K{\t}^{2}+\lambda\,\t-\lambda\,c}{{\t}^{2}}}.
\end{multline}
Naming the coefficients $A$, $B$, $C$, $D$, $p$, $q$ as before,
we now apply Lemma \ref{DElem} to the system \Ref{odes-fin}.
This time the computation of the two quantities used in the lemma is quite laborious,
though still elementary. A symbolic computational program simplifies the result to
the following.
\be
\begin{aligned}
&A(p^2-p')-Bp+C={\frac { \left( \t-c \right) ^{2} \left( (\t-2c)\t\al'
 +2(\t-c)\,\al +2-2\,m \right) }{\t \left( \t-2c \right) }},\\[-4pt]
&D-A(q'-pq)-Bq=0.
\end{aligned}
\end{equation}
By the lemma and the fact $\phi$ is nowhere zero, solutions are only possible if
the first expression vanishes identically, so that $\al$ solves
\be\lb{alpha}(\t-2c)\t\al'
 +2(\t-c)\,\al +2-2\,m =0.$$
The solutions of this take the form
$$\alpha
=\frac {n-2}\t+\frac {\it C}{
\t \left( \t-2c \right) },\end{equation}
where $C$ is a constant. As \Ref{RH} and the second of Equations \Ref{odes} imply that the
form of $\al$ determines that of $\gamma$, we have
the following outcome. If $c=0$, the metric $g$ is conformal to a gradient Ricci soliton
\cite[Proposition 2.4]{confsol}, while if $c\ne 0$ then $g$ is conformal to a quasi-Einstein
metric \cite{quEin}\footnote{See $(2.3)$ in that paper, where the quasi-Einstein
case is given by $\al=(n-2)/\t+a/(\t(1+k\t))$, where $a$ is a constant and $k=-1/2c$. This corresponds
to formula \Ref{alpha} with $C=-2ac$.}.
(The case $C=0$ is a special case of both these types, where $g$ is
conformally Einstein \cite{dr-ma1}.)

To rule out the case that $\wht{g}$ is a nontrivial gradient Ricci soliton, we
note first that the expression defining $\al$ in \Ref{confK}, when compared to
that in \Ref{alpha}, results in
$$(f\circ\ell)'-k\ell'/\ell=\frac C{\t(\t-2c)}.$$ Additionally,
Equation \Ref{coef} can also be written as
$$(f\circ l)''-k\ell''/\ell+2((f\circ\ell)'-k\ell'/\ell )\t^{-1}=0.$$
Substituting the first of these equations in the last term of the second,
and combining the result with the derivative of the first equation gives,
after eliminating $(f\circ\ell)''-k\ell''/\ell$ and rearranging terms
$$k\ell'^2/\ell^2=\frac{2C}{\t^2(\t-2c)}+\left(\frac C{\t(\t-2c)}\right)'
=-\frac{2cC}{\t^2(\t-2c)^2}.$$
Hence the Ricci soliton case $c=0$ implies that $\ell$ is constant, so that
comparing the two expressions for $\al$ again yields $C=0$, i.e. that $\wht{g}$
is Einstein, which is, of course, a special case of the quasi-Einstein condition.

\end{proof}



\section{Appendix: Killing vector fields of the form $w=\t^2\nab f$}\lb{appen}
\setcounter{equation}{0}
We consider here the classification problem for Killing fields of the form of $w=\t^2\nab f$,
a form that played an important role in \S\ref{Kvec}.
In the following $\t$ and $f$ will denote smooth functions on a given manifold.
\begin{Proposition}
On a compact manifold, a Killing field of the form $w=\t^2\nab f$ must be trivial.
\end{Proposition}
\begin{proof}
First, on a compact manifold $\nab f$ has zeros, hence so does $w$.
Let $p$ be a zero of $w=\t^2\nab f$. Since $\nab w=2\t d\t\otimes\nab f+\t^2\nab df$, and at a zero
either $\t=0$ or $\nab f=0$, we see that at a zero $\nab w$ either
equals either zero or $\t^2\nab df$. But in the latter case $\nab w$ is symmetric, yet it
is also skew-symmetric as $w$ is a Killing field, hence $\nab w$ must be zero in this case as well.
However, a Killing field $w$ is uniquely determined by the values of $w$ and $\nab w$ at one point.
As those values are zero at $p$, we see that $w$ must be the zero vector field.
\end{proof}

WIthout compactness, we have the following classification for such vector fields.
\begin{Theorem}
A Riemannian metric $g$ with a Killing vector field of the form $w=\t^2\nab f$ is, near generic
points, a warped product with a one dimensional fiber. If $g$ is also K\"ahler, it is, near such points,
a Riemannian product of a K\"ahler metric with a surface metric admitting a nontrivial Killing vector field.
\end{Theorem}
We note here that a surface with a nontrivial Killing vector field can be presented
as a warped product with a one dimensional fiber and base.
\begin{proof} First, the orthogonal complement $\cal{H}$ to $\mathrm{span}(w)$
is generically $[\nabla f ]^\perp$, which is obviously integrable.
Next, $\cal{H}$ is totally geodesic. This follows immediately since $g(\dot x, w)$
is constant for any geodesic $x(t)$ and Killing field $w$. Alternatively, it
can also be shown directly. With $x$, $y$ denoting vector fields (taking values) in $\cal{H}$,
we compute that $g(\nab_x y,w)=
-g(y,\nab_x w)=-g((\nab w)^*(y),x)=g(\nab_y w,x)=-g(w,\nab_yx))$, where in the penultimate
step we used the Killing property $(\nab w)^*=-\nab w$.
One concludes that the sum $\nab_xy+\nab_yx$ is in $\cal{H}$, and since the same
holds for $\nab_xy-\nab_yx$ by integrability, we see that $\nab_xy$ is in $\cal{H}$.

By a result originating in works of Hiepko \cite{hiep}  along with Ponge and Reckziegel \cite{pon-rec}
(see especially Theorem $3.1$ in the survey of Zeghib \cite{zeg}) a metric is a warped product if
and only if it admits two orthogonal foliations, one totally geodesic and the other
spherical. In our case we have just shown the foliation orthogonal to $w$ is totally geodesic.
The fibers tangent to $\mathrm{span}(w)$, on the other hand, are certainly totally umbilic, as they
are one dimensional. This is part of the definition of spherical. The other part is that the mean
curvature vector is parallel with respect to the normal connection. We now check
this.

Let $w'=w/|w|$ be a unit vector parallel to $w$, defined away from its zeros. The mean curvature
vector to the fibers is then, by definition, $n=\nabla_{w'} w'$, which takes values in $\cal{H}$.
The requirement that $\mathrm{span}(w)$ be spherical amounts to showing that for any
$x\in\cal{H}$, we have $g(\nabla_w n,x)=0$.
The flow of $w$ certainly preserves itself (as $[w,w]=0$) and also $g$ and $\nabla$ (as $w$ is Killing). Therefore the
flow also preserves $w' = w/\sqrt{g(w,w)}$ and thus also $n=\nabla_{w'} w'$. Hence $[w,n]=0$, so that
\begin{multline}
2g(\nabla_w n,x) = 2g(\nabla_n w,x) = g(\nabla_n w,x) - g(n,\nabla_x w)\\\nonumber
= - g(w, \nabla_n x) + g(w,\nabla_x n) = g(w, [x,n]) = 0,
\end{multline}
as $\cal{H}$ is integrable. This concludes the first part of the proof.

What remains is to classify K\"ahler warped products with a one dimensional fiber. Suppose the
manifold is given by $M=B\times F$, with $F$ the
fiber (an interval). Since the base foliation corresponding to $B$ is totally geodesic,
parallel transport along one of its leaves with respect to $g$ is the same as parallel
transport with respect to the induced metric on this leaf, and therefore it
preserves the tangent spaces to these leaves. It is well-known that it also preserves the normal spaces
to the leaves; for completeness, we show explicitly that the unit vector field
$w'$ perpendicular to the leaves is preserved. If $x$ and $y$ are, as usual, vector
fields tangent to the leaves, then $g(w',y)=0$, so $0=d_xg(w',y)=g(\nab_xw',y)+g(w',\nab_xy)=
g(\nab_x w',y)$ because the leaves are totally geodesic, and similarly
$0=d_xg(w',w')=2g(\nab_xw',w')$. So $\nab_x w'$, being orthogonal to a basis, is zero, i.e
$w'$ is parallel in directions tangent to the leaves.

As $g$ is K\"ahler, the complex structure $J$ commutes with any $\nab_x$, so that $Jw'$ is also parallel
in leaf directions. But $Jw'$ is itself tangent to leaves of the base foliation. Therefore, by the
local de Rham Theorem, the induced metric on any leaf splits locally into a Riemannian product
so that $B=N \times I$, where the one dimensional factor $I$ is
tangent to $Jw'$, and $N$ is $J$-invariant, hence has holomorphic (and totally geodesic)
leaves in $M$.

Armed with this information it remains to show that, near generic points,
$$\text{$g$ is a product of a K\"ahler metric
on $N$ and a local metric of revolution on $I\times F$.}$$
For this we turn to a computation that is based on the formulas (cf. \cite{onl})
for the connection of the warped product metric $g=g_B+l^2g_F$,
where the function $l$ is a (lift of) a function on $B$.
Let $t$ be a nontrivial vector field tangent to $F$ which is projectable onto $F$.
Let $s=Jt$, a vector field tangent to $I$.  Then standard
formulas for warped products give
\be\lb{tt}\nab_t t= (\nab_t t)^B+(\nab_t t)^F  = -|t|^2\nab(\log(l))+ct,\end{equation}
with $c$ some function, and the last term takes that form because the fiber is one dimensional.
Next, as $s$ is tangent to $I$, there
is some function $h$ on $M$ such that the vector field
$hs$ is projectable onto I. Therefore, again by warped
product formulas,
\be\lb{tfs}\nab_t (hs) =hs(\log(l))t.\end{equation}
But $\nab_t (hs) = (d_th)s+h\nab_t s= (d_th)s+hJ\nab_t t
= (d_th)s-h|t|^2J\nab(\log(l))+hcs$, by \Ref{tt}.
Equating this expression with the right hand side of
\Ref{tfs} and taking  components tangent to
$N$ gives $h|t|^2[J\nab(\log(l))]^N=0$, so that, away
from the zeros of $h$ and $t$, $[J\nab(\log(l))]^N=0$. Now each tangent
space $T_pN$ is $J$-invariant, so $J$ commutes with the projection to $N$.
Hence $\nab(\log(l))^N=0$ and so $\nab(\log(l))$ is parallel to $s$, which
means that the warping function $l$ is constant on the leaves of $N$, and only changes along the
fibers associated with $I$. Thus $g$ is a Riemannian product of the type claimed above.
\end{proof}


\end{document}